\documentclass{amsart}

\usepackage{xypic}
\input xy
\xyoption{all}
\usepackage{epsfig}
\usepackage{amsthm}
\usepackage{amssymb}
\usepackage{amsmath}
\usepackage{mathabx}
\usepackage{amscd}
\usepackage{color}

\newcommand{\lb}{\varLambda}

\newcommand{\A}{\mathcal{A}}
\def \<{\langle}
\def \>{\rangle}

\newcommand{\bg}{\begin{equation}}
\newcommand{\ed}{\end{equation}}
\newcommand{\bga}{\begin{eqnarray}}
\newcommand{\eda}{\end{eqnarray}}
\newcommand{\pf}{\textbf{Proof:\ }}

\def\cbdu{\par{\raggedleft$\Box$\par}}

\newtheorem {Theorem}  {Theorem}

\numberwithin{Theorem}{section}

\newtheorem {Lemma}[Theorem]  {Lemma}

\theoremstyle{definition}
\newtheorem{Definition}[Theorem]{Definition}
\theoremstyle{remark}

\def \l{\lambda}
\expandafter\chardef\csname pre amssym.def
at\endcsname=\the\catcode`\@ \catcode`\@=11
\def\undefine#1{\let#1\undefined}
\def\newsymbol#1#2#3#4#5{\let\next@\relax
 \ifnum#2=\@ne\let\next@\msafam@\else
 \ifnum#2=\tw@\let\next@\msbfam@\fi\fi
 \mathchardef#1="#3\next@#4#5}
\def\mathhexbox@#1#2#3{\relax
 \ifmmode\mathpalette{}{\m@th\mathchar"#1#2#3}%
 \else\leavevmode\hbox{$\m@th\mathchar"#1#2#3$}\fi}
\def\hexnumber@#1{\ifcase#1 0\or 1\or 2\or 3\or 4\or 5\or 6\or 7\or 8\or
 9\or A\or B\or C\or D\or E\or F\fi}

\font\teneufm=eufm10 \font\seveneufm=eufm7 \font\fiveeufm=eufm5
\newfam\eufmfam
\textfont\eufmfam=\teneufm \scriptfont\eufmfam=\seveneufm
\scriptscriptfont\eufmfam=\fiveeufm

\catcode`\@=\csname pre amssym.def at\endcsname

\newcounter{remark}
\def \grad {\nabla}

\renewcommand{\k}{\kappa}
\renewcommand{\l}{\lambda}

\newcommand{\R}{\mathbf{R}}

\def  \R   {{\mathbb R}}
\def  \Z   {{\mathbb Z}}

\def  \T   {{\mathbb T}}

\def  \12  {{\frac{1}{2}}}

\def\build#1_#2^#3{\mathrel{\mathop{\kern 0pt#1}\limits_{#2}^{#3}}}

\begin{document}

\title[Determining Modes for 3D NSE]{Determining modes for the \\ 3D Navier-Stokes equations}

\author [Alexey Cheskidov]{Alexey Cheskidov}
\address{Department of Mathematics, Stat. and Comp.Sci.,  University of Illinois Chicago, Chicago, IL 60607,USA}
\email{acheskid@uic.edu} 
\author [Mimi Dai]{Mimi Dai}
\address{Department of Applied Mathematics, Stat. and Comp.Sci.,  University of Illinois Chicago, Chicago, IL 60607,USA}
\email{mdai@uic.edu} 
\author [Landon Kavlie]{Landon Kavlie}
\address{Department of Mathematics, Stat. and Comp.Sci.,  University of Illinois Chicago, Chicago, IL 60607,USA}
\email{lkavli2@uic.edu}

\thanks{The authors were partially supported by NSF Grant
DMS--1108864 and DMS--1517583.}





\begin{abstract}
We introduce a determining wavenumber for the forced 3D Navier-Stokes equations (NSE)  defined for each individual solution.
Even though this wavenumber blows up if the solution blows up, its time average is uniformly bounded for all solutions on the weak global attractor.
The bound is compared to Kolmogorov's dissipation wavenumber and the Grashof constant.
 
\bigskip


\hspace{0.02cm}CLASSIFICATION CODE: 35Q35, 37L30.
\end{abstract}

\maketitle

\section{Introduction}

The Navier-Stokes equations (NSE) on a torus,  are given by 
\begin{equation}
  \label{nse}
  \left\{
    \begin{array}{l}
      u_t + (u \cdot \grad) u - \nu\Delta u + \grad p = f \\
      \grad \cdot u = 0,
    \end{array}
  \right.
\end{equation}
where $u$ is the velocity, $p$ is the pressure, and $f$ is the external force. We assume that
$f$ has zero mean, and consider zero mean solutions. 

The dissipative nature of these equations is reflected in the existence of an absorbing ball in $L^2$.
Moreover, in the two-dimensional case, there exists a compact global attractor which uniformly attracts all bounded subsets of $L^2$. This attracting set is, in fact, finite dimensional, as was proven by
Foias and Temam in \cite{FT} (see also \cite{CFT}). The first result for the finite dimensionality of a two-dimensional fluid appeared in the work of Foias and Prodi \cite{FP}, where they showed that high modes of
a solution are controlled by low modes asymptotically as time goes to infinity. The number of these
low modes, called determining modes, was estimated by Foias, Manley, Temam and Treve \cite{FMTT} and later improved by Jones and Titi  \cite{JT}. See also
\cite{CJT, FJKT,FT84, FTiti} and references therein for related results.

In three dimensions the situation is drastically different as the equations have thus far eluded a proof for the existence of classical solutions. Even so, the existence of a global attractor for weak solutions is known in a weak sense \cite{FT-attractor,FMRT} (see also \cite{Se} for a related notion of a trajectory attractor). This weak global attractor consists of points on complete bounded
trajectories and attracts all bounded subsets of $L^2$ in the weak topology. 
However, it is not known whether the solutions on the attractor are regular, unless the attractor consists of a single fixed point. Neither is it known whether the attractor is compact or finite-dimensional. Similarly, the existence of a finite number of determining modes is not known in the three-dimensional case. Nevertheless, Constantin, Foias, Manley, and Temam  \cite{CFMT} showed the existence of determining modes assuming that the $H^1$ norm of solutions is uniformly bounded. 
The question whether the global attractor of the 3D NSE is bounded in $H^1$ is open and may very well require a resolution of the regularity problem. However, even assuming regularity, this would not immediately guarantee that the $H^1$ bound would depend only on the size of the force (a Grashof constant), and not on the shape of the force.

With no hope of getting a finite number of determining modes for the 3D NSE, one might ask
whether this can be done in some average sense. Indeed, the Kolmogorov 41  phenomenological 
theory of turbulence \cite{K41} predicts that the number of degrees of freedom should be of order
$\kappa_\mathrm{d}^3$, where $\kappa_\mathrm{d}$ is Kolmogorov's dissipation wavenumber.
This number is often used as the resolution needed for direct numerical simulations, so one
might ask an alternative question: What is the number of determining modes for a time discretization
of the 3D NSE and how does it depend on the force and time step?

In this paper, without making any assumptions regarding regularity properties of solutions or
bounds on the global attractor, we prove the existence of a time-dependent determining
wavenumber $\lb_u(t)$
defined for each individual solution $u$. We show that any weak solutions on the global attractor $u$ and $v$ that coincide below $\max\{\lb_u, \lb_v\}$ have to be identical. The wavenumber $\lb_u(t)$ 
blows up if and only if the solution $u(t)$ blows up. Nevertheless, the time average of this wavenumber
is uniformly bounded on the global attractor, which we estimate in terms of the Kolmogorov dissipation
number and Grashof constant.

To begin, let $u$ be a weak solution of the 3D Navier-Stokes equations. For $r\in(2,3)$ we define a local determining wavenumber
\[
\lb_{u,r}(t):=\min\{\lambda_q:\lambda_{p}^{-1+\frac 3r}\|u_p\|_{L^r}<c_r\nu ,~\forall p>q~\text{and}~ \lambda_q^{-1}\|u_{\leq q}\|_{L^\infty}<c_r\nu,~q\in \mathbb{N} \},
\]
where $c_r$ is an adimensional constant that depends only on $r$ (it will be choosen
in Section~\ref{sec:pf}).
Here $\lambda_q =\frac{2^q}{L}$, $L$ is the size of the torus, $u_q = \Delta_q u$ is the Littlewood-Paley projection of $u$, and $u_{\leq q}=\sum_{-1\leq p\leq q}u_p$ (see Section~\ref{sec:pre}). Notice that $\lb_{u,r}(t)$ might not be finite as
we adopt a convention that $\min{\emptyset}= \infty$. 

Thanks to Bernstein's inequality, we have
\[
\lb_{u,r} \geq \lb^{\mathrm{dis}}_{u}:=\min\{\lambda_q:\lambda_{p}^{-1}\|u_p\|_{L^\infty}<c_0\nu ,~\forall p>q,~q\in \mathbb{N} \},
\]
which is a local dissipation wavenumber introduced by Cheskidov and Shvydkoy in \cite{CSr}. It defines a dissipation range where a local Reynolds number corresponding
to high frequencies is small, i.e.,
\[
\mathcal{R}^h_q:= \frac{l_q \|u_q\|_{L^\infty}}{\nu} < c_0, \qquad \forall \l_q >  \lb^{\mathrm{dis}}_{u},
\]
where $l_q = \l_q^{-1}$. The dominance of the dissipation term above  $\lb^{\mathrm{dis}}_{u}$ is reflected  in improved Beale-Kato-Majda and Prodi-Serrin criteria where
$u$ is replaced with its projection on modes below  $\lb^{\mathrm{dis}}_{u}$ (see \cite{CSr}).
In particular, $\lb^{\mathrm{dis}}_{u}$  (and hence $\lb_{u,r}$) blows up if and only if the solution $u$ blows up.
 The determining wavenumber $\lb_{u,r}$
imposes tougher condition on high modes, as well as requires a control on low modes via the low frequency Reynolds number
\[
\mathcal{R}^l_q:=\frac{l_q \|u_{\leq q}\|_{L^\infty}}{\nu} < c_r, \qquad \l_q = \lb_{u,r}.
\]
It is also worth mentioning that a similar determining wavenumber is used in \cite{CD} to prove the existence of a finite number of determining modes for the surface quasi-geostrophic
equation equation in critical and subcritical cases. Even though
the determining wavenumber enjoys uniform bounds in those cases, it still proved useful to start with a time dependent wavenumber defined based on the structure of
the equation only, and then study its dependence on the force using available bounds for the global attractor.
 
We prove the following.
\begin{Theorem}\label{thm}
Let $u(t)$ and $v(t)$ be weak solutions of the 3D Navier-Stokes equations.  Let $\lb(t):=\max\{\lb_{u,r}(t), \lb_{v,r}(t)\}$ for some $r\in(2,3)$. Let $Q(t)$ be such that $\lb(t)=\lambda_{Q(t)}$.
If
\begin{equation} \label{eq:dm-condition}
u(t)_{\leq Q(t)}=v(t)_{\leq Q(t)}, \qquad \forall t>0,
\end{equation}
then
\begin{equation}\notag
\lim_{t \to \infty} \|u(t) - v(t)\|_{L^2}=0.
\end{equation}
\end{Theorem}
The proof of this theorem, given in Section~\ref{sec:pf}, also implies the following results:
\begin{Theorem}\label{cor}
If $u(t)$ and $v(t)$ are two Leray-Hopf solutions on the weak global attractor $\A$ such that
\begin{equation} \label{eq:dm-condition}
u(t)_{\leq Q(t)}=v(t)_{\leq Q(t)}, \qquad \forall t<0,
\end{equation}
where $Q$ is given in Theorem~\ref{thm}, then
\[
u(t) = v(t), \qquad \forall t \leq 0.
\]   
\end{Theorem}
This establishes the existence of determining modes for the 3D Navier-Stokes equations. It is worthwhile to note that the determining wavenumber $\lb_{u,r}$ depends on time and may
not be bounded. Actually, it is bounded if and only if $u$ is regular. However, the average determining wavenumber $\<\lb\> = \frac{1}{T}\int_{t}^{t+T} \lb_{u,r}(\tau)\, d\tau$
always enjoys a  uniform bound. Indeed, we establish the following pointwise bound:
\begin{equation} \label{pestimate}
\lb_{u,r}(t) \lesssim_r \frac{\| \nabla u(t)\|_{L^2}^2}{\nu^2}.
\end{equation}
We direct the readers to Section \ref{sec:pre} for the  $\lesssim_r$ notation. 
Note that this automatically provides a finite number of determining modes and recovers the results by Constantin, Foias, Manley, and Temam \cite{CFMT} in the 
case where $\| \nabla u(t)\|_{L^2}^2$ is bounded on the global attractor, which is known for small forces (laminar regimes). On the other hand, \eqref{pestimate} holds in
general for arbitrary forces and implies that $\<\lb\>$ is uniformly bounded for all Leray-Hopf solutions on the global attractor, i.e., complete bounded trajectories. However,
the bound \eqref{pestimate} is sharp only in the case of extreme intermittency, where on average there is only one eddy at each dyadic scale. To make a connection with Kolmogorov's
turbulence theory \cite{K41, F} we have to define an intermittency dimension and analyze $\<\lb\>$ in various intermittency regimes.

In Section~\ref{sec:Kolmogorov} and in Section~\ref{sec:grashof}, we further examine $\<\lb\>$, comparing it to Kolmogorov's dissipation wavenumber as well as the Grashof constant, defined as

\[
\kappa_\mathrm{d} := \left(\frac{\varepsilon }{\nu^3} \right)^{\frac{1}{d+1}}, \qquad G:=\frac{\|f\|_{H^{-1}}}{\nu^2 \lambda_0^{1/2}}, \qquad  \varepsilon := \lambda_0^{d}\nu\<\|\nabla u\|_{L^2}^2\>,
\]
where $d\in[0,3]$ is the intermittency dimension.This parameter is defined in Section~\ref{sec:Kolmogorov} in terms of the level of saturation of Bernstein's inequality. The
case $d=3$, where there is no intermittency and  eddies occupy the whole region, corresponds to Kolmogorov's regime. In this case the bounds read
\[
\<\lb \> \lesssim_\delta   \kappa_d^{2+\delta} \lambda_0^{-1-\delta} \lesssim_\delta \lambda_0 \left(\frac{G^2}{T\nu^2\lambda_0^2 }  +  G^2     \right)^{\frac{1}{2} + \delta}, \qquad d=3,
\]
where $\delta$ can be arbitrary small when $r$ is chosen close to $3$. On the other hand, in the case of extreme intermittency, the bounds become
\[
\<\lb \> \lesssim \kappa_d \lesssim \frac{G^2}{T\nu^2\lambda_0 }  +  \lambda_0 G^2 , \qquad d=0.
\]

 Finally, in Section~\ref{app} we show that if one of the weak solutions in Theorems~\ref{thm}, \ref{cor} satisfies the energy equality, then $\lb$ can be defined in terms of that solution only.
The energy equality holds for regular solutions, such as steady states, or solutions belonging to Onsager's space $L^3(0,\infty; B^{1/3}_{3,\infty})$ (see \cite{CCFS} where
the proof can be applied to the viscous case, where the space $B^{1/3}_{3,c_0}$ can be replaced with  $B^{1/3}_{3,\infty}$).


\bigskip

\section{Preliminaries}
\label{sec:pre}

\subsection{Notation}
\label{sec:notation}
We denote by $A\lesssim B$ an estimate of the form $A\leq C B$ with
some absolute constant $C$, by $A\sim B$ an estimate of the form $C_1
B\leq A\leq C_2 B$ with some absolute constants $C_1$, $C_2$, and by $A\lesssim_r B$ an estimate of the form $A\leq C_r B$ with
some adimentional constant $C_r$ that depends only on the parameter $r$. $L^p$ and 
$H^s$ stand for Lebesgue and $L^2$-based Sobolev spaces respectively. To simplify
the notations we will write $\|\cdot\|_p=\|\cdot\|_{L^p}$, and use $(\cdot, \cdot)$ for the $L^2$-inner product.

\subsection{Definition of solutions}

\begin{Definition} \label{d:weaksolution}
A Leray-Hopf solution of \eqref{nse} on $[0,T]$ is a function
$u:[0,T] \to L^2(\R^3)$ in the class
\[
u \in  C_{\mathrm{w}}([0,T];L^2(\T^3)) \cap L^2_{\rm loc}([0,T];H^1(\T^3)),
\]
satisfying $\nabla \cdot u =0$ in the sense of distributions,
\[
(u(t),\varphi(t)) + \int_0^t \left[ - (u,\partial_t \varphi) + \nu (\nabla u,\nabla \varphi) + (u\cdot \nabla u, \varphi) + (f, \varphi )\right] \, ds = (u_0,\varphi(0)),
\]
for all $t\in[0,T]$ and all test functions
$\varphi \in C^\infty([0,T]\times \T^3)$ with $\nabla \cdot \varphi = 0$, and
satisfying the energy inequality
\begin{equation}\label{SEI}
{\textstyle \frac12}\|u(t)\|_2^2 \leq {\textstyle \frac12}\|u(t_0)\|_2^2 +\int_{t_0}^t
\left[- \|\nabla u(s)\|_2^2 +(f,u)\right]\, ds,
\end{equation}
for almost all $t_0 \in (0,T)$ and all $t \in [t_0,T]$.
\end{Definition}

Even though the global existence of Leray-Hopf solutions (for all $T>0$) is known
(first proved by Leray in 1934 in the case of the whole space \cite{L34}), the uniqueness remains an open problem. 

By a standard  approximation argument, the function $\varphi$ in the definition of weak solution can be taken to be only weakly Lipschitz in time (see \cite{RRS}).
This will allow us to use Littlewood-Paley projections of $u(t)$ as 
test functions.

\subsection{Littlewood-Paley decomposition}
\label{sec:LPD}
The techniques presented in this paper rely strongly on the Littlewood-Paley decomposition. Thus we here recall the Littlewood-Paley decomposition theory briefly. For a more detailed description on this theory we refer the readers to the books by Bahouri, Chemin and Danchin \cite{BCD} and Grafakos \cite{Gr}. 

We denote $\lambda_q=\frac{2^q}{L}$ for integers $q$. A nonnegative radial function $\chi\in C_0^\infty(\R^3)$ is chosen such that 
\begin{equation} \label{eq:xi}
\chi(\xi):=
\begin{cases}
1, \ \ \mbox { for } |\xi|\leq\frac{3}{4}\\
0, \ \ \mbox { for } |\xi|\geq 1.
\end{cases}
\end{equation}
Let 
\[
\varphi(\xi):=\chi(\xi/2)-\chi(\xi)
\]
and
\begin{equation}\label{eq-phiq}
\varphi_q(\xi):=
\begin{cases}
\varphi(2^{-q}\xi)  \ \ \ \mbox { for } q\geq 0,\\
\chi(\xi) \ \ \ \mbox { for } q=-1,
\end{cases}
\end{equation}
so that the sequence of $\varphi_q$ forms a dyadic partition of unity. Given a tempered distribution vector field $u$ on $\T^3 =[0,L]^3$ and $q \ge -1$, an integer, the $q$th Littlewood-Paley projection of $u$ is given by 
\begin{equation} \label{eq:Def_delta_q}
  u_q(x) := \Delta_q u(x) := \sum_{k\in\Z^3}\widehat{u}(k)\varphi_q(k)e^{i\frac{2\pi}{L} k \cdot x},
\end{equation}
where $\widehat{u}(k)$ is the $k$th Fourier coefficient of $u$.
Note that $u_{-1} = \widehat{u}(0)$. Then
\[
u=\sum_{q=-1}^\infty u_q
\]
in the distributional sense. Note that
\[
  \|u\|_{H^s} \sim \left(\sum_{q=-1}^\infty\lambda_q^{2s}\|u_q\|_2^2\right)^{1/2},
\]
for each $u \in H^s$ and $s\in\R$. To simplify the notations, we denote
\bg\notag
u_{\leq Q}:=\sum_{q=-1}^Qu_q = \sum_{k\in\Z^3}\widehat{u}(k)\chi(2^{-Q-1}k)e^{i\frac{2\pi}{L} k \cdot x}.
\ed

Next we prove the following property for the dyadic blocks.

\begin{Lemma}\label{le-dyadic}
For any $Q\geq -1$, the identity 
\[u_Q=(\Delta_Q+\Delta_{Q+1})u_{\leq Q}\]
holds.
\end{Lemma}
\pf
We denote \[\psi_Q:=\sum_{-1\leq q\leq Q}\varphi_q.\]
It follows from (\ref{eq:xi}) and (\ref{eq-phiq}) that
$\psi_Q(\xi)=\chi(2^{-Q-1}\xi)$. By \eqref{eq:Def_delta_q}, in order to prove the lemma, it is sufficient to show that
\begin{equation}\notag
\varphi_Q=\psi_Q(\varphi_Q+\varphi_{Q+1}).
\end{equation}
Indeed, we have from (\ref{eq:xi}) and (\ref{eq-phiq}),
\[\varphi_Q(\xi)+\varphi_{Q+1}(\xi)=\chi(2^{-Q-2}\xi)-\chi(2^{-Q}\xi).\]
Thus,
\begin{equation}\notag
\begin{split}
\psi_Q(\xi)\left[\varphi_Q(\xi)+\varphi_{Q+1}(\xi)\right]=\chi(2^{-Q-1}\xi)\left[\chi(2^{-Q-2}\xi)-\chi(2^{-Q}\xi)\right].
\end{split}
\end{equation}
It then follows from the facts  
\begin{equation}\notag
\chi(2^{-Q-1}\xi)\chi(2^{-Q-2}\xi)=\chi(2^{-Q-1}\xi), \qquad
\chi(2^{-Q-1}\xi)\chi(2^{-Q}\xi)=
\chi(2^{-Q}\xi),
\end{equation}
that 
\[\psi_Q(\xi)\left[\varphi_Q(\xi)+\varphi_{Q+1}(\xi)\right]=\chi(2^{-Q-1}\xi)-\chi(2^{-Q}\xi)=\varphi_Q(\xi),\]
which completes the proof.

\cbdu

\subsection{Bony's paraproduct, Littlewood-Paley theorem, and inequalities.}
\label{sec-para}

Here we state some useful properties for the dyadic blocks of the Littlewood-Paley decomposition and some inequalities that will be used often through the paper.

We start with Bony's paraproduct formula. First, note that
\[
u\cdot\nabla v = \sum_{p} u_{\leq{p-2}}\cdot\nabla v_p + \sum_{p} u_{p}\cdot\nabla v_{\leq{p-2}} + \sum_p \sum_{|p-p'|\leq 1} u_p\cdot\nabla v_{p'}.
\]
Due to \eqref{eq:xi} we have $\varphi(\xi)=0$ when $|\xi|\leq 3/4$ or $|\xi|\geq 2$, and hence 
\[
(f_q  g_{\leq q-2})_{\geq q+2} \equiv 0, \qquad (f_q  g_{\leq q-2})_{\leq q-3} \equiv 0, \qquad (f_qg_{q+1})_{\geq q+3} \equiv 0,
\]
for tempered distributions $f$ and $g$. Therefore,
\begin{equation}\notag
\begin{split}
\Delta_q(u\cdot\nabla v)=&\sum_{q-1\leq p \leq q+2}\Delta_q(u_{\leq{p-2}}\cdot\nabla v_p)+
\sum_{q-1\leq p \leq q+2}\Delta_q(u_{p}\cdot\nabla v_{\leq{p-2}})\\
&+\sum_{p\geq q-2}\sum_{\substack{|p-p'|\leq 1\\ p' \geq q-2}}\Delta_q(u_p\cdot\nabla v_{p'}).
\end{split}
\end{equation}
It is usually sufficient to use a weaker form of this formula:
\begin{equation}\notag
\begin{split}
\Delta_q(u\cdot\nabla v)=&\sum_{|q-p|\leq 2}\Delta_q(u_{\leq{p-2}}\cdot\nabla v_p)+
\sum_{|q-p|\leq 2}\Delta_q(u_{p}\cdot\nabla v_{\leq{p-2}})\\
&+\sum_{p\geq q-2} \Delta_q(\tilde u_p \cdot\nabla v_p),
\end{split}
\end{equation}
where $\tilde{u}_p := u_{p-1} + u_p + u_{p+1}$.

Finally, we will take advantage of the Littlewood-Paley Theorem, which yields
the following bounds on the Lebesgue norms:
\begin{equation} \label{eq:LPT}
\left(\sum_q \|u_q\|_r^r\right)^{\frac{1}{r}} \lesssim \|u\|_r \lesssim
\left(\sum_q \|u_q\|_2^2\right)^{\frac{1}{2}}, \qquad r\geq 2.
\end{equation}

We recall Bernstein's inequality.
\begin{Lemma}\label{le:bern}
Let $n$ be the spacial dimension and $r\geq s\geq 1$. Then for all tempered distributions $u$, 
\bg\label{Bern}
\|u_q\|_{r}\lesssim \lambda_q^{n(\frac{1}{s}-\frac{1}{r})}\|u_q\|_{s}.
\ed
\end{Lemma}
We will also use Jensen's inequality stated bellow.
\begin{Lemma}\label{le:bern}
Let $\varphi: I\to \R$ be a convex function on an interval $I$. For any points $x_1, x_2, ..., x_n\in I$ and weights $\theta_1, \theta_2, ..., \theta_n\in(0,1)$, the inequality  
\bg\label{Jen}
\varphi(\theta_1x_1+\theta_2x_2+\cdot\cdot\cdot+\theta_nx_n)\leq \theta_1\varphi(x_1)+\theta_2\varphi(x_2)+\cdot\cdot\cdot+\theta_n\varphi(x_n).
\ed
holds. 
\end{Lemma}

\bigskip

\section{Proof of Theorems \ref{thm} and \ref{cor}}
\label{sec:pf}

Let $w:=u-v$, which satisfies the equation
\begin{equation} \label{eq-w}
w_t+u\cdot\nabla w+w\cdot\nabla v=-\nabla \pi+\nu \Delta w
\end{equation}
in the distribution sense (see Definition~\ref{d:weaksolution}).
By our assumption, $w_{\leq Q(t)}(t)\equiv 0$.

Using $\Delta^2_qw$ as a test function for (\ref{eq-w}) and adding up for all $q\geq -1$ yields  
\begin{equation}\label{w2}
\begin{split}
&\frac{1}{2}\|w(t)\|_2^2- \frac{1}{2}\|w(t_0)\|_2^2 +\nu \int_{t_0}^t\|\nabla w\|_2^2 \, d\tau\\
\leq & \int_{t_0}^t \sum_{q\geq -1} \left| \int_{\T^3}\Delta_q(w\cdot\nabla v) w_q\, dx\right| \, d\tau
+\int_{t_0}^t\sum_{q\geq -1}\left| \int_{\T^3}\Delta_q(u\cdot\nabla w) w_q\, dx\right|\, d\tau,\\
=:&\int_{t_0}^t I_1\, d\tau+ \int_{t_0}^t I_2\, d\tau.
\end{split}
\end{equation}

 Using Bony's paraproduct mentioned in Subsection \ref{sec-para} and the triangle inequality, $I_1$ is decomposed as
\begin{equation}\notag
\begin{split}
I_1\leq&\sum_{q\geq -1}\sum_{|q-p|\leq 2}\left|\int_{\T^3}\Delta_q(w_{\leq{p-2}}\cdot\nabla v_p) w_q\, dx\right|\\
&+\sum_{q\geq -1}\sum_{|q-p|\leq 2}\left|\int_{\T^3}\Delta_q(w_{p}\cdot\nabla v_{\leq{p-2}})w_q\, dx\right|\\
&+\sum_{q\geq -1}\sum_{p\geq q-2} \left|\int_{\T^3}\Delta_q(\tilde w_p\cdot\nabla v_p)w_q\, dx\right|\\
=:&I_{11}+I_{12}+I_{13}.
\end{split}
\end{equation}
These terms are estimated as follows. Recall that $w_{\leq Q(t)} \equiv 0$. Then using H\"older's inequality we obtain
\[
\begin{split}
  I_{11} 
& \leq \sum_{p\geq Q+3}\sum_{|q-p|\leq 2}\int_{\T^3}|\Delta_q(w_{\leq{p-2}}\cdot\nabla v_p) w_q|\, dx\\
 &\lesssim  \sum_{p \geq Q+3 }\sum_{|q-p|\leq 2}\|w_{\leq p-2}\|_{\frac{2r}{r-2}}\lambda_p\|v_p\|_r\|w_q\|_2.
\end{split}
\]
Now using definition of $\lb_{v,r}$, Bernstein's inequality (\ref{Bern}), Young's inequality and Jensen's inequality (\ref{Jen}), we get 
\[
\begin{split}
 I_{11}  &\lesssim  c_r\nu\sum_{p\geq Q+3}\sum_{|q-p|\leq 2}\lambda_p^{2-\frac 3r}\|w_q\|_2\sum_{p'\leq p-2}\|w_{p'}\|_{\frac{2r}{r-2}} \\
           &\lesssim  c_r\nu\sum_{q\geq Q+1}\lambda_q\|w_q\|_2\sum_{p'\leq q}\lambda_{p'}\|w_{p'}\|_2\lambda_{q-p'}^{1-\frac 3r}\\
            &\lesssim  c_r\nu\sum_{q\geq Q+1}\lambda_q^2\|w_q\|_2^2+c_r\nu\sum_{q\geq Q+1}\left(\sum_{p'\leq q}\lambda_{p'}\|w_{p'}\|_2\lambda_{q-p'}^{1-\frac 3r}\right)^2\\
             &\lesssim  c_r\nu\sum_{q\geq Q+1}\lambda_q^2\|w_q\|_2^2+c_r\nu\sum_{q\geq Q+1}\sum_{p'\leq q}\lambda_{p'}^2\|w_{p'}\|_2^2\lambda_{q-p'}^{1-\frac 3r}\\
              &\lesssim  c_r\nu\sum_{q\geq Q+1}\lambda_q^2\|w_q\|_2^2+c_r\nu\sum_{-1\leq p'\leq Q}\lambda_{p'}^2\|w_{p'}\|_2^2\sum_{q\geq Q}\lambda_{q-p'}^{1-\frac 3r}
              +c_r\nu\sum_{p'\geq Q}\lambda_{p'}^2\|w_{p'}\|_2^2\sum_{q\geq p'}\lambda_{q-p'}^{1-\frac 3r}\\
           &\lesssim  c_r\nu \|\nabla w\|_2^2,
\end{split}
\]
where we used the fact that $r<3$.

Next, we have
\[
\begin{split}
  I_{12} &\leq \sum_{p> Q}\sum_{|q-p|\leq 2}\int_{\T^3}|\Delta_q(w_p\cdot\nabla v_{\leq p-2}) w_q| \, dx\\
&\leq  \sum_{p> Q}\sum_{|q-p|\leq 2}\int_{\T^3}|\Delta_q(w_p\cdot\nabla v_{\leq Q}) w_q| \, dx+\sum_{p'> Q}\sum_{p\geq p'+2}\sum_{|q-p|\leq 2}\int_{\T^3}|\Delta_q(w_p\cdot\nabla v_{p'}) w_q| \, dx\\
           &\lesssim  \|\nabla v_{\leq Q}\|_\infty\sum_{q> Q}\sum_{|q-p|\leq 2} \|w_p\|_2\|w_q\|_2+\sum_{p'> Q}\|\nabla v_{p'}\|_\infty\sum_{p\geq p'+2}\sum_{|q-p|\leq 2}\| w_p\|_2\|w_q\|_{2}\\
           &\lesssim  \|\nabla v_{\leq Q}\|_\infty\sum_{q> Q}\|w_q\|^2_2+\sum_{p'>Q}\lambda_{p'}^{1+\frac{3}{r}}\|v_{p'}\|_r\sum_{q\geq p'}\|w_q\|^2_{2}.
\end{split}
\]
Using the definition of $\lb_{v,r}$ and changing the order of summation, we obtain
\[
\begin{split}
  I_{12}         &\lesssim  c_r\nu\sum_{q> Q}\lambda_Q\lambda_q\| w_q\|_2^2
+c_r\nu\sum_{p'> Q}\lambda_{p'}^{2}\sum_{q\geq p'}\| w_q\|_2^2\\
           &\lesssim  c_r\nu\sum_{q> Q}\lambda_q^2\| w_q\|_2^2
+c_r\nu\sum_{q> Q}\lambda_q^2\| w_q\|_2^2\sum_{Q< p'\leq q}\lambda_{q-p'}^{-2}\\
           &\lesssim  c_r\nu \|\nabla w \|_2^2,
\end{split}
\]
where $\sum_{Q< p'\leq q}\lambda_{q-p'}^{-2}\lesssim 1$ is used in the last inequality.

We will now  estimate $I_{13}$. Recall that $w_{\leq Q} \equiv 0$ and hence $\tilde w_{\leq Q-1} \equiv 0$. Since $\Delta_q$ and $\nabla$ commute, integration by parts and changing the order
of summation yield
\begin{equation}\notag
\begin{split}
I_{13}=&\sum_{q\geq -1}\sum_{p\geq q-2} \left|\int_{\T^3}\Delta_q(\tilde w_p\cdot\nabla v_p)\cdot w_q\, dx\right|\\
=& \sum_{q>Q}\sum_{p\geq q-2}\left|\int_{\T^3}\Delta_q(\tilde w_p \otimes v_{p}) :\nabla w_q\, dx\right|\\
= & \sum_{p\geq Q}\sum_{Q<q\leq p+2}\left|\int_{\T^3}\Delta_q(\tilde w_p\otimes v_{p}) :\nabla w_q \, dx\right|.\\
\end{split}
\end{equation}
Thanks to H\"older's inequality, we have
\[
\begin{split}
  I_{13} &=  \sum_{p> Q}\sum_{Q<q\leq p+2}\left|\int_{\T^3}\Delta_q(\tilde w_p\otimes v_{p}) :\nabla w_q \, dx\right|
+  \sum_{Q<q\leq Q+2}\left|\int_{\T^3}\Delta_q(\tilde w_Q\otimes v_{Q}) :\nabla w_{q} \, dx\right|\\
           &\lesssim  \sum_{p> Q}\|\tilde w_p\|_2\|v_{p}\|_r\sum_{Q<q\leq p+2}\lambda_q\|w_q\|_{\frac{2r}{r-2}}
+\lambda_Q\sum_{Q<q\leq Q+2} \|\tilde w_Q\|_2\|v_{Q}\|_\infty \|w_{q}\|_2.\\
\end{split}
\]
By Lemma \ref{le-dyadic} and definition of $\Lambda_{v,r}$, 
\[
\|v_{Q}\|_\infty  =\|(\Delta_{Q}+\Delta_{Q+1})v_{\leq Q}\|_\infty \lesssim \|v_{\leq Q}\|_\infty \leq c_r\nu \l_Q.
\]
Hence,  it follows
\[
\begin{split}
I_{13}    &\lesssim c_r \nu \sum_{p>Q}\lambda_p^{1-\frac 3r}\|\tilde w_{p}\|_2\sum_{Q<q\leq p+2}\lambda_q^{1+\frac 3r}\|w_q\|_2
+c_r\nu\lambda_Q^2\sum_{Q<q\leq Q+2} \|\tilde w_Q\|_2 \|w_q\|_2\\
           &\lesssim c_r \nu\sum_{p> Q-1}\lambda_p\|w_p\|_2\sum_{Q<q\leq p+2}\lambda_q\|w_q\|_2\lambda_{q-p}^{\frac 3r}
+c_r\nu\lambda_Q^2\sum_{Q<q\leq Q+2}\|w_q\|^2_2.
\end{split}
\]
Then we apply Young's inequality and Jensen's inequality (\ref{Jen}) to deduce
\begin{equation}\notag
\begin{split}
I_{13}\lesssim & c_r \nu\sum_{p> Q-1}\lambda_p^2\|w_p\|_2^2+ c_r \nu\sum_{p> Q-1}\left(\sum_{Q<q\leq p+2}\lambda_q\|w_q\|_2\lambda_{q-p}^{\frac 3r}\right)^2\\
&+c_r\nu\lambda_Q^2\sum_{Q<q\leq Q+2}\|w_q\|^2_2\\
\lesssim & c_r \nu\sum_{p> Q-1}\lambda_p^2\|w_p\|_2^2+ c_r \nu\sum_{p> Q-1}\sum_{Q<q\leq p+2}\lambda_q^2\|w_q\|_2^2\lambda_{q-p}^{\frac 3r}\\
&+c_r\nu\lambda_Q^2\sum_{Q<q\leq Q+2}\|w_q\|^2_2\\
\lesssim & c_r \nu\sum_{p> Q-1}\lambda_p^2\|w_p\|_2^2+ c_r \nu\sum_{q>Q}\lambda_q^2\|w_q\|_2^2\sum_{p\geq q-2}\lambda_{q-p}^{\frac 3r}\\
&+c_r\nu\lambda_Q^2\sum_{Q<q\leq Q+2}\|w_q\|^2_2\\
\lesssim&c_r \nu \|\nabla w\|_2^2.
\end{split}
\end{equation}
Combining the estimates above leads to  
\begin{equation}\label{est-i1}
I_1\lesssim c_r \nu \|\nabla w\|_2^2.
\end{equation}

To estimate $I_2$, we start with the decomposition
\begin{equation}\notag
\begin{split}
I_2\leq&\sum_{q\geq -1}\sum_{|q-p|\leq 2}\left|\int_{\T^3}\Delta_q(u_{\leq{p-2}}\cdot\nabla w_p)w_q\, dx\right|\\
&+\sum_{q\geq -1}\sum_{|q-p|\leq 2}\left|\int_{\T^3}\Delta_q(u_{p}\cdot\nabla w_{\leq{p-2}})w_q\, dx\right|\\
&+\sum_{q\geq -1}\sum_{p\geq q-2}\sum_{|p-p'|\leq 1}\left|\int_{\T^3}\Delta_q(u_p\cdot\nabla w_{p'})w_q\, dx\right|\\
=:& I_{21}+I_{22}+I_{23}.
\end{split}
\end{equation}
Then we estimate the terms as follows:
\[
\begin{split}
  I_{21} &=  \sum_{p> Q}\sum_{|q-p|\leq 2}\left|\int_{\T^3}\Delta_q(u_{\leq{p-2}}\cdot\nabla w_p)w_q\, dx\right|\\
  &\leq\sum_{p >Q}\sum_{|q-p|\leq 2}\int_{\T^3}|\Delta_q(u_{\leq Q}\cdot\nabla w_p)w_q|\, dx
  +\sum_{p'>Q}\sum_{p\geq p'+2}\sum_{|q-p|\leq 2}\int_{\T^3}|\Delta_q(u_{p'}\cdot\nabla w_p) w_q|\, dx\\
 &\lesssim  \|u_{\leq Q}\|_\infty \sum_{p>  Q} \sum_{|q-p|\leq 2} \lambda_p \|w_p\|_2\|w_q\|_2  
 +\sum_{p'> Q}\|u_{p'}\|_\infty\sum_{p\geq p'+2}\sum_{|q-p|\leq 2}\lambda_p\|w_p\|_2\|w_q\|_{2}\\
           &\lesssim  \|u_{\leq Q}\|_\infty\sum_{q> Q}\lambda_q \|w_q\|^2_2+\sum_{p'>Q}\lambda_{p'}^{\frac{3}{r}}\|u_{p'}\|_r\sum_{q\geq p'}\lambda_q \|w_q\|^2_{2}.
\end{split}
\]
Using the definition of $\lb_{u,r}$ and changing the order of summation, we arrive at
\[
\begin{split}
  I_{21}         &\lesssim  c_r\nu\sum_{q> Q}\lambda_Q\lambda_q\| w_q\|_2^2
+c_r\nu\sum_{p'> Q}\lambda_{p'}\sum_{q\geq p'}\lambda_q \| w_q\|_2^2\\
           &\lesssim  c_r\nu\sum_{q> Q}\lambda_q^2\| w_q\|_2^2
+c_r\nu\sum_{q> Q}\lambda_q^2\| w_q\|_2^2\sum_{Q< p'\leq q}\lambda_{q-p'}^{-1}\\
&\lesssim  c_r\nu\|\nabla w\|_2^2,
\end{split}
\]
where we used $\sum_{Q< p'\leq q}\lambda_{q-p'}^{-1}\lesssim 1$ for the last inequality.

We proceed to estimating the second term. First, H\"older's inequality and Bernstein's inequality (\ref{Bern}) yield
\[
\begin{split}
  I_{22} &= \sum_{q>Q}\sum_{|q-p|\leq 2}\left|\int_{\T^3}\Delta_q(u_{p}\cdot\nabla w_{\leq{p-2}})w_q\, dx\right|\\
&\lesssim  \sum_{q>Q}\sum_{|q-p|\leq 2}\|u_p\|_r\|w_q\|_2\sum_{p'\leq q-2}\lambda_{p'}\|w_{p'}\|_{\frac{2r}{r-2}}\\
           &\lesssim  c_r\nu \sum_{q>Q}\lambda_q^{1-\frac 3r}\|w_q\|_2\sum_{p'\leq q-2}\lambda_{p'}^{1+\frac 3r}\|w_{p'}\|_2\\
           &\lesssim  c_r\nu \sum_{q>Q}\lambda_q\|w_q\|_2\sum_{p'\leq q-2}\lambda_{p'}\|w_{p'}\|_2\lambda_{q-p'}^{-\frac 3r}.
\end{split}
\]
We continue by using Young's inequality, Jensen's inequality (\ref{Jen}), and changing the order of summation to obtain
\begin{equation}\notag
\begin{split}
I_{22}\lesssim & c_r\nu \sum_{q>Q}\lambda_q^2\|w_q\|_2^2+c_r\nu \sum_{q>Q}\left(\sum_{p'\leq q-2}\lambda_{p'}\|w_{p'}\|_2\lambda_{q-p'}^{-\frac 3r}\right)^2\\
\lesssim & c_r\nu \sum_{q>Q}\lambda_q^2\|w_q\|_2^2+c_r\nu \sum_{q>Q}\sum_{p'\leq q-2}\lambda_{p'}^2\|w_{p'}\|_2^2\lambda_{q-p'}^{-\frac 3r}\\
\lesssim & c_r\nu \sum_{q>Q}\lambda_q^2\|w_q\|_2^2+c_r\nu \sum_{p'\geq -1}\lambda_{p'}^2\|w_{p'}\|_2^2\sum_{q\geq p'-2}\lambda_{q-p'}^{-\frac 3r}\\
\lesssim & c_r \nu \|\nabla w\|_2^2.
\end{split}
\end{equation}

The last term $I_{23}$ is similar to $I_{13}$:
\[
\begin{split}
 I_{23} &= \sum_{q>Q}\sum_{p\geq q-2}\left|\int_{\T^3}\Delta_q(u_p\otimes \tilde w_p) \nabla w_q \, dx\right|\\
&\leq  \sum_{p> Q}\sum_{Q<q\leq p+2}\int_{\T^3}|\Delta_q(u_p \otimes\tilde w_p) \nabla w_q| \, dx
+  \sum_{Q<q\leq Q+2}\int_{\T^3}|\Delta_q(u_Q \otimes \tilde w_Q) \nabla w_{q}| \, dx\\
           &\lesssim  \sum_{p> Q}\|u_{p}\|_r\|\tilde w_p\|_2\sum_{Q<q\leq p+2}\lambda_q\|w_q\|_{\frac{2r}{r-2}}
+\lambda_Q\sum_{Q<q\leq Q+2} \|u_{Q}\|_\infty \|\tilde w_Q\|_2\|w_{q}\|_2\\
&\lesssim c_r \nu \sum_{p>Q}\lambda_p^{1-\frac 3r}\|\tilde w_{p}\|_2\sum_{Q<q\leq p+2}\lambda_q^{1+\frac 3r}\|w_q\|_2
+c_r\nu\lambda_Q^2\sum_{Q<q\leq Q+2} \|\tilde w_Q\|_2 \|w_q\|_2\\
           &\lesssim c_r \nu\sum_{p> Q}\lambda_p\|\tilde w_p\|_2\sum_{Q<q\leq p+2}\lambda_q\|w_q\|_2\lambda_{q-p}^{\frac 3r}
+c_r\nu\lambda_Q^2\sum_{Q<q\leq Q+2}\|w_q\|^2_2\\
           &\lesssim c_r \nu \|\nabla w\|_2^2.
\end{split}
\]

Therefore, we have 
\begin{equation}\label{est-i2}
I_2\lesssim c_r \nu \|\nabla w\|_2^2.
\end{equation}

Combining \eqref{est-i1} and \eqref{est-i2}, we conclude that there exists an adimensional constant $C>0$ that depends only on $r$, such that
\[
I_1+I_2 \leq C c_r \nu \|\nabla w\|_2^2.
\]
Choosing $c_r:=1/(2C)$ we infer from \eqref{w2} that
\[
\begin{split}
\|w(t)\|_2^2 \leq & \|w(t_0)\|_2^2 - \nu \int_{t_0}^t \|\nabla w\|_2^2 \, d\tau\\
\leq & \|w(t_0)\|_2^2 - \nu \k_0^2 \int_{t_0}^t \|w\|_2^2 \, d\tau,
\end{split}
\]
where $\k_0=\frac{2\pi}{L}$.
Thus 
\[
\|w(t)\|_2^2 \leq \|w(t_0)\|_2^2e^{-\nu\k_0^2(t-t_0)}.
\]
Taking the limit as $t\to \infty$ concludes the proof of Theorem~\ref{thm}, while taking the limit as $t_0 \to -\infty$
concludes the proof of Theorem~\ref{cor}.

\cbdu

\section{Explicit estimates of the average determining wavenumber in terms of Kolmogorov's dissipation wavenumber}
\label{sec:Kolmogorov}

In this section we show that the average determining
wavenumber has a uniform upper bound. Recall that 
\[
\lb_{u,r}(t):=\min\{\lambda_q:\lambda_{p}^{-1+\frac 3r}\|u_p\|_{L^r}<c_r\nu ,~\forall p>q~\text{and}~ \lambda_q^{-1}\|u_{\leq q}\|_{L^\infty}<c_r\nu,~q\in \mathbb{N} \},
\]
where $r \in(2,3)$ and $c_r$ is an adimensional constant that depends only on $r$. To simplify the notations, denote $\lb:=\lb_{u,r}$ and let $Q $ be such that $\lambda_Q = \lb$.

We start with the following observation that will be used to estimate $\lb$.

\begin{Lemma} \label{l:mainbound}
If $\l_0 < \lb < \infty$, then
\[
\lb \leq (c_r\nu)^{-1} \max\{\lb^{\frac3r}\|u_{Q}\|_r, \ 2  \|u_{\leq Q-1}\|_r \}.
\]
If $\lb=\infty$, then $\|u\|_{H^{\frac12}} = \infty$.
\end{Lemma}
\begin{proof}
First consider the case where $\lb$ is finite. 
If both conditions
in the definition of $\lb$ are satisfied for all $q \geq 0$, then $\lb = \l_0$. If $\lb > \l_0$, then both conditions in the definition of $\lb$ are satisfied for
$q=Q$, but
one of the conditions is not satisfied for $q=Q-1$, i.e.,
\[
\l_Q^{-1+\frac3r}\|u_{Q}\|_r \geq c_r\nu, \qquad \text{or} \qquad \|u_{\leq Q-1}\|_\infty \geq c_r \nu\lambda_{Q-1}={\textstyle \frac{1}{2}}c_r \nu\lb,
\]
and the conclusion of the lemma holds. 

On the other hand, if $\lb= \infty$, then for every
$q \in \mathbb{N}$ either
\[
\sup_{p>q} \l_q^{-1+\frac3r}\|u_p\|_r \geq c_r\nu, \qquad \text{or} \qquad \l_q^{-1}\|u_{\leq q}\|_\infty \geq c_r \nu.
\]
If
\[
\limsup_{p \to \infty} \l_q^{-1+\frac3r}\|u_p\|_r \geq c_r \nu,
\]
then due to Bernstein's inequality,
\[
\|u\|_{H^{\frac12}} \sim \sum_q\l_q\|u_q\|^2_2 \gtrsim\sum_q\l_q^{-2+\frac6r}\|u_q\|^2_r=\infty,
\]
and we are done. Otherwise, there exists $q^{**}$
such that
\[
\l_q^{-1}\|u_{\leq q}\|_\infty \geq c_r \nu, \qquad \forall q\geq q^{**}.
\]
Now if $\|u\|_{H^{\frac12}} < \infty$, then for any $\epsilon >0$ there exists $q^*\geq q^{**}$ such that $\|u_{> q^*}\|_{H^{\frac12}}\leq\epsilon$. Then for $q\geq q^*$ we use Bernstein's inequality,
Jensen's inequality, and \eqref{eq:LPT} to obtain
\[
\begin{split}
\l_q^{-1}\|u_{\leq q}\|_\infty &\leq \sum_{p\leq q} \l_q^{-1} 
\l_p^{\frac32}\|u_p\|_2\\
&\lesssim \sum_{p\leq q^*\leq q} \l_{q*-q}^{\frac12}\l_{p-q}^{\frac12} \l_p^{\frac12}\|u_p\|_2 + \sum_{q^*<p\leq q} \l_{p-q} \l_p^{\frac12} \|u_p\|_2\\
&\lesssim \l_{q^*-q}^{\frac12}\left(\sum_{p\leq q^*\leq q} \l_{p-q}^{\frac12} \l_p\|u_p\|_2^2\right)^{\frac{1}{2}}
+ \left(\sum_{q^*<p\leq q} \l_{p-q} \l_p\|u_p\|_2^2\right)^{\frac{1}{2}}
\\&\lesssim \l_{q^*-q}^{\frac12}\|u_{\leq q^*}\|_{H^{\frac12}} + \|u_{>q^*}\|_{H^{\frac12}}\\
\\&\leq \l_{q^*-q}^{\frac12}\|u_{\leq q^*}\|_{H^{\frac12}} + \epsilon.
\end{split}
\]
Passing to the limit as $q\to \infty$ in the above inequality leads to 
\[c_r\nu\leq \epsilon,\]
for arbitrarily small $\epsilon>0$, which is a contradiction. Thus we conclude that $\|u\|_{H^{\frac12}}=\infty$.
\end{proof}

As the first consequence, we will show that $\lb(t)$ is
locally integrable for every Leray-Hopf solution $u(t)$.

\begin{Lemma} \label{l:simplebound}
The determining wavenumber $\lb$ enjoys the following bound:
\[
\lb - \l_0 \lesssim \frac{\|\nabla u\|_2^2}{\nu^2}.
\]
\end{Lemma}
\begin{proof}
If $\lb$ is infinite then $\|\nabla u\|_2$ is also infinite thanks to Lemma~\ref{l:mainbound}. Now consider the case where $\l_0<\lb<\infty$. Then according to Lemma~\ref{l:mainbound},
\[
\lb \leq (c_r\nu)^{-1}\lb^{\frac3r} \|u_{Q}\|_r, \qquad \text{or} \qquad \lb \leq 2(c_r\nu)^{-1} \|u_{\leq Q-1}\|_\infty.
\]
In the first case we use Bernstein's inequality to obtain
\[
\begin{split}
\lb &\leq (c_r\nu)^{-2} \lb^{\frac6r-1} \|u_{Q}\|^2_r\\
&\lesssim (c_r\nu)^{-2} \lb^2 \|u_{Q}\|^2_2 \\
&\lesssim (c_r\nu)^{-2} \| \nabla u\|_2^2.
\end{split}
\]
In the second case,
\[
\begin{split}
\lb &\leq 4(c_r\nu)^{-2} \lb^{-1} \|u_{\leq Q-1}\|^2_\infty \\
&\lesssim (c_r\nu)^{-2} \lb^{-1}\left(\sum_{q<Q}\|u_q\|_\infty \right)^2 \\
&\lesssim (c_r\nu)^{-2} \lb^{-1}\left(\sum_{q<Q}\lambda_q^{\frac{3}{2}}\|u_q\|_2 \right)^2\\
&= (c_r\nu)^{-2} \left(\sum_{q<Q}\lambda_q\|u_q\|_2 2^{\frac{1}{2}(q-Q)} \right)^2\\
&\lesssim (c_r\nu)^{-2} \sum_{q<Q}\lambda_q^2\|u_q\|_2^2 2^{\frac{1}{2}(q-Q)}\\
&\lesssim (c_r\nu)^{-2} \| \nabla u\|_2^2,
\end{split}
\] 
and the conclusion of the lemma holds again.

\end{proof}

We can now compare the average determining wavenumber with Kolmogorov's dissipation wavenumber,  often defined as
\[
\kappa_\mathrm{d} := \left(\frac{\varepsilon }{\nu^3} \right)^{\frac{1}{4}},
\]
where $\varepsilon$ is the average energy dissipation rate
\begin{equation} \label{eq:kd-classical}
\varepsilon := \frac{\nu}{L^3} \<\|\nabla u\|_2^2\> = \frac{\nu}{TL^3}\int_t^{t+T} \|\nabla u(\tau)\|_2^2 \, d\tau,
\end{equation}
and $\<\cdot \>$ denotes the time average. Then Lemma~\ref{l:simplebound} yields the following bound:
\begin{equation} \label{eq:first-bound-on-average-lambda}
\<\lb\>-\l_0 \lesssim \frac{\<\|\nabla u\|_2^2\>}{\nu^2} = \frac{\varepsilon}{\nu^3}L^3 = \kappa_\mathrm{d}^4 \l_0^{-3}.
\end{equation}
However, in this argument we used Bernstein's inequalities that might not be sharp in a turbulent regime. The level of saturation of Bernstein's inequalities is measured by a parameter $d$, called
the intermittency dimension (see \cite{CSint} where the notions of active regions, eddies, and intermittency are defined mathematically). The number of eddies at scale $\ell$ grows as
\[
\text{Number of eddies} \sim \left(\frac{L}{\ell}\right)^{d}.
\]
The case $d=3$ corresponds to Kolmogorov's regime
where eddies occupy the whole region for each scale in the inertial range. The other extreme
case is $d=0$, where the number of eddies is of order one on all the scales, in which case
Bernstein's inequalities are sharp.  A recent DNS performed by Kaneda et al. \cite{kaneda}  on the Earth Simulator suggests that $d \approx 2.7$.
The presence of intermittency requires the following modification of the classical definition \eqref{eq:kd-classical}:
\begin{equation} \label{eq:kdeps-inermit}
\kappa_\mathrm{d} := \left(\frac{\varepsilon }{\nu^3} \right)^{\frac{1}{d+1}}, \qquad  \varepsilon := \frac{\nu}{L^d}\<\|\nabla u\|_2^2\>,
\end{equation}
where $d\in[0,3]$ is the intermittency dimension. This parameter is chosen  so that
\begin{equation} \label{eq:intermdef}
\left<\sum_{q\leq Q}\l_q^{-1+\frac{6}{r}+d(1-\frac{2}{r})} \|u_q\|_r^2 \right> \lesssim \l_0^{d(1-\frac{2}{r})}\left<\sum_{q\leq Q}\l_q^{2} \|u_q\|_2^2 \right>.
\end{equation}
This can be done since by Bernstein's inequality
\[
\l_0^{3-\frac{6}{r}}\l_q^{-1+\frac{6}{r}} \|u_q\|_2^2 \lesssim  \l_q^{-1+\frac{6}{r}} \|u_q\|_r^2 \lesssim \l_q^{2} \|u_q\|_2^2.
\]
Note that $-1+d(1-\frac{2}{r}) <0$ since $d \leq 3$ and $r< 3$. Therefore,
\begin{equation} \label{eq:estonaver111}
\begin{split}
\lb^{-1+d(1-\frac{2}{r})} \|u_{\leq Q-1}\|_\infty^2 &\lesssim
\lb^{-1+d(1-\frac{2}{r})} \left(\sum_{q<Q} \|u_q\|_\infty\right)^2 \\
&\lesssim \lb^{-1+d(1-\frac{2}{r})} \left(\sum_{q<Q} \l_q^{\frac{3}{r}}\|u_q\|_r\right)^2\\
&=  \left(\sum_{q<Q} \l_q^{-\frac{1}{2}+\frac{3}{r}+\frac{d}{2}(1-\frac{2}{r})}\|u_q\|_r
(L\l_{Q-q})^{-\frac12+\frac d2(1-\frac{2}{r})} \right)^2\\
&\lesssim_r \sum_{q<Q} \l_q^{-1+\frac{6}{r}+d(1-\frac{2}{r})}\|u_q\|_r^2.
\end{split}
\end{equation}
Now, thanks to Lemma~\ref{l:mainbound}, we have
\[
\lb^{-1+\frac{6}{r}+d(1-\frac{2}{r})} \|u_Q\|_r^2 \geq (c_r\nu)^2 \lb^{1+d(1-\frac{2}{r}) }
\]
or
\[
4\lb^{-1+d(1-\frac{2}{r})} \|u_{\leq Q-1}\|_\infty^2 \geq (c_r \nu)^2\lb^{1+d(1-\frac{2}{r})},
\]
provided $\l_0<\lb<\infty$. Therefore,
\[
(c_r \nu)^2\lb^{1+d(1-\frac{2}{r})} \leq \lb^{-1+\frac{6}{r}+d(1-\frac{2}{r})} \|u_Q\|_r^2 + 4\lb^{-1+d(1-\frac{2}{r})} \|u_{\leq Q-1}\|_\infty^2,
\]
whenever $\l_0<\lb<\infty$.
Combining this with \eqref{eq:estonaver111} we obtain
\begin{equation} \label{eq:d<3andnoQ}
\begin{split}
(c_r \nu)^2\lb^{1+d(1-\frac{2}{r})} &\lesssim_r \lb^{-1+\frac{6}{r}+d(1-\frac{2}{r})} \|u_Q\|_r^2 + \sum_{q<Q} \l_q^{-1+\frac{6}{r}+d(1-\frac{2}{r})}\|u_q\|_r^2\\
&=\sum_{q\leq Q} \l_q^{-1+\frac{6}{r}+d(1-\frac{2}{r})}\|u_q\|_r^2.
\end{split}
\end{equation}
Note that when $\lb = \infty$ the bound \eqref{eq:d<3andnoQ}
holds as well since both sides in these inequalities are
infinite due to Lemma~\ref{l:mainbound}.
Taking the time average of \eqref{eq:d<3andnoQ} and using \eqref{eq:intermdef}, we arrive at
\[
\begin{split}
(c_r \nu)^2\left<\lb^{1+d(1-\frac{2}{r})}\right>_{>\l_0} &\lesssim_r
\left<\sum_{q\leq Q} \l_q^{-1+\frac{6}{r}+d(1-\frac{2}{r})}\|u_q\|_r^2\right>_{>\l_0}\\
& \lesssim  \l_0^{d(1-\frac{2}{r})}\left<\sum_{q\leq Q}\l_q^{2} \|u_q\|_2^2 \right>_{>\l_0}\\
& \lesssim  \l_0^{d(1-\frac{2}{r})}\<\|\nabla u\|_2^2\>\\
&=:  \l_0^{-\frac{2d}{r}}\frac{\varepsilon}{\nu},
\end{split}
\]
where $\< g \>_{>\l_0} = \< 1_{g>\l_0} g \>$.
Finally, using Jensen's inequality,
\begin{equation} \label{eq:bound-on-lambda-using-intermit}
\begin{split}
\<\lb \>-\l_0= \<\lb \>_{>\l_o}&=  \left(\frac{(c_r \nu)^2\<\lb\>_{>\l_0}^{1+d(1-\frac{2}{r})}}{(c_r\nu)^2} \right)^{\frac{1}{1+d(1-\frac{2}{r})}}\\
&\lesssim  \left(\frac{(c_r \nu)^2\<\lb^{1+d(1-\frac{2}{r})}\>_{>\l_0}}{(c_r\nu)^2} \right)^{\frac{1}{1+d(1-\frac{2}{r})}}\\
&\lesssim_r  \left(\frac{\varepsilon }{\nu^3} \right)^{\frac{1}{1+d(1-\frac{2}{r})}} \l_0^{-\frac{2d}{r+d(r-2)}}.\\
\end{split}
\end{equation}
Comparing it with Kolmogorov's dissipation wavenumber \eqref{eq:kdeps-inermit}, we get
\[
\<\lb \>-\l_0 \lesssim_r \kappa_\mathrm{d}^\frac{1+d}{1+d(1-\frac{2}{r})}\l_0^{-\frac{2d}{r+d(r-2)}}.
\]
In case of extreme intermittency $d=0$ the powers do not depend on $r$, so we can just choose $r=5/2$ inferring that 
the average determining wavenumber is bounded by  Kolmogorov's dissipation wavenumber $\kappa_d$ :
\[
\<\lb \>-\l_0\lesssim  \frac{\varepsilon }{\nu^3}=\kappa_\mathrm{d} , \qquad \text{when} \qquad d=0.
\]
 However, in Kolmogorov's regime $d=3$, the bound becomes $\kappa_\mathrm{d}^{2+}$ since $r$ can only be taken less than $3$. More precisely, for any $d\leq 3$ we
have the following bound:
\[
\<\lb \>-\l_0 \lesssim_r \kappa_\mathrm{d}^\frac{2r}{2r- 3}\l_0^{-\frac{3}{2r-3}} =\kappa_\mathrm{d}^{2+\delta}\l_0^{-1-\delta},
\]
where $\delta=\frac{6-2r}{2r-3}$ can be arbitrarily small when $r$ is chosen close to $3$.

\section{Explicit estimates of the average determining wavenumber in terms of the Grashof number}
\label{sec:grashof}
It is well known that the 3D Navier-Stokes equation possesses an absorbing ball in $L^2$
\[
B:= \{ u\in L^2(\T^3): \|u\|_2 \leq R \},
\]
where $R$ is any number larger than $\frac{\|f\|_{H^{-1}}}{\nu \kappa_0}$ and $\kappa_0=2\pi\l_0 =2\pi/L$, which can be expressed in terms of the adimensional
Grashof number
\[
G:=\frac{\|f\|_{H^{-1}}}{\nu^2 \kappa_0^{1/2}}
\]
as $R > \nu \kappa_0^{-1/2} G$. Then for any Leray solution $u(t)$ there exists $t_0$, depending only on $\|u(0)\|_2$, such that
\[
u(t) \in B \qquad \forall t>t_0. 
\]
Then the evolutionary system consisting of Leray-Hopf solutions in the absorbing ball posseses a weak global attractor $\A$, which has the following structure \cite{FT,FMRT}:
\[
\A = \{u(0): u(\cdot) \text{ is a complete bounded  Leray-Hopf solution to the NSE}\}.
\]
The set $\A \subset B$ is the minimal $L^2$ weakly closed weakly attracting set, it is $L^2$-weak omega limit of $B$ (see \cite{C,CF}), and $\|u\|_2 \leq  \nu \kappa_0^{-1/2} G$ for all $u \in \A$. See also \cite{FRT} for topological properties of $\A$.

Let $u(t)$ be a trajectory on the global attractor $\A$. To bound the average determining wavenumber in terms of the Grashof number we use the energy inequality:
\[
\begin{split}
0\leq\|u(t_0+t)\|_2^2 &\leq \limsup_{t\to t_0+}\|u(t)\|_2^2 - 2\nu\int_{t_0}^{t_0+T} \|\nabla u(t)\|_2^2\, dt + 2\int_{t_0}^{t_0+T} (f,u)\, dt\\
&\leq  \nu^2 \kappa_0^{-1} G^2 - \nu\int_{t_0}^{t_0+T} \|\nabla u(t)\|_2^2\, dt + \frac{1}{\nu}\int_{t_0}^{t_0+T} \|f\|_{H^{-1}}^2\, dt.\\
\end{split}
\]
Therefore
\[
\begin{split}
\frac{1}{T}\int_{t_0}^{t_0+T} \|\nabla u(t)\|_2^2\, dt &\leq \frac{\nu G^2}{T\kappa_0} + \frac{\|f\|^2_{H^{-1}}}{\nu^2}\\
&\leq \frac{\nu G^2}{T\kappa_0} + \kappa_0\nu^2G^2.
\end{split}
\]
Then \eqref{eq:first-bound-on-average-lambda} imply
\begin{equation} \label{eq:bound-on-Lambda-G}
\<\lb\>-\l_0 \lesssim \frac{\<\|\nabla u\|_2^2\>}{\nu^2} \leq  \frac{G^2}{\nu T \kappa_0} + \kappa_0G^2.
\end{equation}
To take into account intermittency, we can use \eqref{eq:bound-on-lambda-using-intermit} instead of \eqref{eq:first-bound-on-average-lambda} to obtain
\[
\begin{split}
\<\lb\>-\l_0 &\lesssim_{r} \left(\frac{\varepsilon }{\nu^3} \right)^{\frac{1}{1+d(1-\frac{2}{r})}} \l_0^{-\frac{2d}{r+d(r-2)}}\\
&= \left(\frac{\l_0^{d(1-\frac{2}{r})}\<\|\nabla u\|_2^2\>   }{\nu^2} \right)^{\frac{1}{1+d(1-\frac{2}{r})}}\\ 
& \leq \kappa_0 \left(\frac{G^2}{\nu T\kappa_0^2 }  +  G^2     \right)^{\frac{1}{1+d(1-\frac{2}{r})}},
\end{split}
\]
where $d \in[0, 3)$ is the intermittency parameter from Section~\ref{sec:Kolmogorov}. In the case of extreme intermittency $d=0$ this bound is the same as \eqref{eq:bound-on-Lambda-G}, proportional to
$G^2$. However, in Kolmogorov's regime where $d=3$, the average determining number is bounded by $G^{1+}$. 
More precisely, 
\[
\begin{split}
\<\lb\>-\l_0 &\lesssim_\delta \kappa_0 \left(\frac{G^2}{\nu T\kappa_0^2 }  +  G^2     \right)^{\frac{1}{2} + \delta}, \qquad d=3,
\end{split}
\]
where $\delta =\frac{3-r}{4r-6} \to 0$ as $r \to 3-$.

\section{The case of energy equality}
\label{app}
In this section we consider the case where one of the solutions of the 3D Navier-Stokes equations satisfies the energy equality, e.g., a steady state or a solution belonging to
Onsager's space $L^3(0,\infty; B^{1/3}_{3,\infty})$.

\begin{Theorem}\label{thm-unique}
Let $v(t)$ be a weak solution of the 3D NSE satisfying the energy equality, and $Q(t)$ be such that $\lb_{v,r}(t)=\lambda_{Q(t)}$ for some $r\in(2,3)$.
If $u(t)$ is a Leray-Hopf weak solution such that
\begin{equation} \label{eq:dm-condition}
u(t)_{\leq Q(t)}=v(t)_{\leq Q(t)}, \qquad \forall t>0,
\end{equation}
then
\begin{equation}\notag
\lim_{t \to \infty} \|u(t) - v(t)\|_{L^2}=0.
\end{equation}
\end{Theorem}
\begin{proof}
We know that $v$ satisfies energy equality 
\begin{equation}\label{energy1}
\int_{\T^3}\frac12v^2\, dx+\nu \int_0^t\int_{\T^3} |\nabla v|^2\, dx\, d\tau=
\int_{\T^3}\frac12v_0^2\, dx+\int_0^t\int_{\T^3} fv\, dx\, d\tau;
\end{equation}
and $u$ satisfies energy inequality
\begin{equation}\label{energy2}
\int_{\T^3}\frac12u^2\, dx+\nu \int_0^t\int_{\T^3} |\nabla u|^2\, dx\, d\tau\leq
\int_{\T^3}\frac12u_0^2\, dx+\int_0^t\int_{\T^3} fu\, dx\, d\tau;
\end{equation}
Computing the energy of the difference $w:=u-v$
\begin{equation}\notag
\begin{split}
&\int_{\T^3}\frac12|w|^2\, dx+\nu \int_0^t\int_{\T^3} |\nabla w|^2\, dx\, d\tau\\
= & \int_{\T^3}\frac12u^2\, dx+\int_{\T^3}\frac12v^2\, dx
+\nu \int_0^t\int_{\T^3} |\nabla u|^2\, dx\, d\tau+\nu \int_0^t\int_{\T^3} |\nabla v|^2\, dx\, d\tau\\
&-\int_{\T^3}uv\, dx-2\nu \int_0^t\int_{\T^3} \nabla u\nabla v\, dx\, d\tau,
\end{split}
\end{equation}
combining (\ref{energy1}) and (\ref{energy2}) we conclude
\begin{equation}\label{energy3}
\begin{split}
&\int_{\T^3}\frac12|w|^2\, dx+\nu \int_0^t\int_{\T^3} |\nabla w|^2\, dx\, d\tau\\
\leq & \int_{\T^3}\frac12(u_0^2+v_0^2)\, dx+\int_0^t\int_{\T^3} f(u+v)\, dx\, d\tau\\
&-\int_{\T^3}uv\, dx-2\nu \int_0^t\int_{\T^3} \nabla u\nabla v\, dx\, d\tau.
\end{split}
\end{equation}
As for the $L^2$ inner product of $u$ and $v$, we have (see, e.g., \cite{T}) 
\begin{equation}\label{energy4}
\begin{split}
&  \int_{\T^3}u_0v_0\, dx-\int_{\T^3}uv\, dx-2\nu \int_0^t\int_{\T^3} \nabla u\nabla v\, dx\, d\tau\\
=&-\int_0^t\int_{\T^3}w\cdot\nabla v w\,dx\, d\tau-\int_0^t\int_{\T^3} f(u+v)\, dx\, d\tau.
\end{split}
\end{equation}
Combining (\ref{energy3}) and (\ref{energy4}) yields
\begin{equation}\notag
\begin{split}
\int_{\T^3}\frac12|w|^2\, dx+\nu \int_0^t\int_{\T^3} |\nabla w|^2\, dx\, d\tau
\leq  \int_{\T^3}\frac12|w(0)|^2\, dx+\int_0^t\int_{\T^3}|w\cdot\nabla v w|\,dx\, d\tau.
\end{split}
\end{equation}
It then follows from estimate \eqref{est-i1} that for a small constant $c_r$
\begin{equation}\notag
\begin{split}
\int_{\T^3}|w|^2\, dx+\nu \int_0^t\int_{\T^3} |\nabla w|^2\, dx\, d\tau
\leq  \int_{\T^3}|w(0)|^2\, dx.
\end{split}
\end{equation}
Applying Poincar\'e's inequality, we have
\begin{equation}\notag
\begin{split}
\int_{\T^3}|w|^2\, dx \leq  \int_{\T^3}|w(0)|^2\, dx-\kappa_0\nu \int_0^t\int_{\T^3} |w|^2\, dx\, d\tau.
\end{split}
\end{equation}
Thus
\[
\|w(t)\|_2^2\leq \|w(0)\|_2^2  e^{-\kappa_0\nu t}.
\]
\end{proof}

\section*{Acknowledgment}
The authors thank the anonymous referee for careful reading the manuscript and constructive comments.


\begin{thebibliography}{XX}



\bibitem{BCD}
H. Bahouri, J. Chemin,  and R. Danchin.
\newblock {\em Fourier analysis and nonlinear partial differential equations}.
\newblock Grundlehrender Mathematischen Wissenschaften, 343. Springer, Heidelberg, 2011.


\bibitem{C}
A. Cheskidov.
\newblock {\em Global attractors of evolutionary systems}.
\newblock Journal of Dynamics and Differential Equations, 21: 249--268, 2009.

\bibitem{CCFS}
A. Cheskidov, P. Constantin, S. Friedlander and R. Shvydkoy.
\newblock {\em Energy conservation and Onsager's conjecture for the Euler equations}.
\newblock Nonlinearity, 21: 1233--1252, 2008.

\bibitem{CD}
A. Cheskidov and M. Dai.
\newblock {\em Determining modes for the surface Quasi-Geostrophic equation}.
\newblock arXiv:1507.01075, 2015.

\bibitem{CF}
A. Cheskidov and C. Foias.
\newblock {\em On global attractors of the 3D Navier-Stokes equations}.
\newblock J. Differential Equations, Vol. 231, (2): 714--754, 2006.




\bibitem{CSr}
A. Cheskidov and R. Shvydkoy.
\newblock {\em A unified approach to regularity problems for the 3D Navier-Stokes and Euler equations: the use of Kolmogorov's dissipation range}.
\newblock J. Math. Fluid Mech., 16:263--273, 2014.

\bibitem{CSint}
A. Cheskidov and R. Shvydkoy.
\newblock {\em Euler Equations and Turbulence: Analytical Approach to Intermittency}.
\newblock SIAM J. Math. Anal., 46(1), 353--374, 2014.

\bibitem{CJT}
B. Cockburn, D. Jones,  and E. Titi.
\newblock{\em  Estimating the number of asymptotic degrees of freedom for nonlinear dissipative systems}.
\newblock  Mathematics of Computation, Vol. 66 (219): 1073--1087, 1997. 


\bibitem{CF}
P. Constantin and C. Foias.
\newblock{\em Navier-Stokes Equations}.
\newblock University of Chicago press. 

\bibitem{CFMT}
P. Constantin, C. Foias, O. P. Manley, and R. Temam.
\newblock{\em  Determining modes and fractal dimension of turbulent flows}.
\newblock  J. Fluid Mech., 150:427--440, 1985. 

\bibitem{CFT}
P. Constantin, C. Foias, and R. Temam.
\newblock{\em On the dimension of the attractors in two-dimensional turbulence}.
\newblock Physica D, 30, 284--296, 1988.


\bibitem{FJKT}
C. Foias, M. Jolly, R. Kravchenko and E. Titi.
\newblock {\em A determining form for the 2D Navier-Stokes equations -- the Fourier modes case}.
\newblock J. Math. Phys., 53(11), 115623, 30 pp, 2012.


\bibitem{FMRT}
C. Foias, O. Manley, R. Rosa, and R. Temam.
\newblock {\em Navier-Stokes equations and turbulence}.
\newblock Vol. 83 of Encyclopedia of Mathematics and its Applications. Cambridge University Press,
Cambridge, 2001.

\bibitem{FMTT}
C. Foias, O. P. Manley, R Temam, and Y. M. Tr\'eve.
\newblock {\em Asymptotic analysis of the Navier-Stokes equations}.
\newblock  Phys. D, 9(1-2), 157--188, 1983.

\bibitem{FP}
C. Foias and G. Prodi.
\newblock {\em Sur le comportement global des solutions non-stationnaires
des \'equations de Navier--Stokes en dimension 2}.
\newblock Rend. Sem. Mat. Univ. Padova 39:1--34, 1967.

\bibitem{FRT}
C. Foias, R. Rosa,  and R. Temam.
\newblock {\em  Topological properties of the weak global attractor of the three-dimensional Navier-Stokes equations}.
\newblock Discrete and Continuous Dynamical Systems, Vol. (4): 1611--1631, 2010. 

\bibitem{FT}
C. Foias and R. Temam.
\newblock {\em  Some analytic and geometric properties of the solutions of the Navier-Stokes equations}.
\newblock J. Math. Pures Appl., 58, 339--368, 1979. 

\bibitem{FT84}
C. Foias and R. Temam.
\newblock {\em Determination of the solutions of the Navier-Stokes equations by a set of nodal values}.
\newblock  Math. Comput., 43, 117--133, 1984.

\bibitem{FT-attractor}
C. Foias and R. Temam.
\newblock {\em The connection between the Navier-Stokes equations,and turbulence theory}.
\newblock Directions in Partial Differential Equations, Publ. Math. Res. Center Univ. Wisconsin, 55-73, 1985.


\bibitem{FTiti}
C. Foias and  E. S. Titi.
\newblock {\em Determining nodes, finite difference schemes and inertial manifolds}.
\newblock Nonlinearity, 135--153, 1991.

\bibitem{F}
U. Frisch.
\newblock {\em Turbulence: The legacy of A. N. Kolmogorov}.
\newblock Cambridge University Press, Cambridge, 1995.

\bibitem{Gr}
L. Grafakos.
\newblock {\em Modern Fourier Analysis}.
\newblock Second edition. Graduate Texts in Mathematics, 250. Springer, New York, 2009.

\bibitem{JT}
D. A. Jones and E. S. Titi.
\newblock {\em Upper bounds on the number of determining modes, nodes, and volume elements for the Navier-Stokes equations}.
\newblock Indiana Univ. Math. J. 42(3):875--887, 1993.




\bibitem{kaneda} 
Y.~Kaneda, T.~Ishihara, M.~Yokokawa, K.~Itakura, and A.~Uno.
\newblock Energy dissipation rate and energy spectrum in high resolution direct numerical simulations of turbulence in a periodic box
\newblock{\em Physics of Fluids}, 15(2):21--24, 2003.

\bibitem{K41}
A. N. Kolmogorov.
\newblock {\em The local structure of turbulence in incompressible viscous fluids
at very large Reynolds numbers}.
\newblock Dokl. Akad. Nauk. SSSR 30: 301?305, 1941.

\bibitem{L34} 
J. Leray.
\newblock{\em Sur le mouvement dun liquide visqueux emplissant lespace}.
\newblock Acta Math., 63(1):193--248, 1934.


\bibitem{RRS}
J. C. Robinson, J. L. Rodrigo, and W. Sadowski.
\newblock {\em The Three-Dimensional Navier-Stokes Equations}.
\newblock Cambridge University Press, 2016.

\bibitem{Se}
G. R. Sell.
\newblock {\em Global attractors for the three-dimensional Navier-Stokes equations}.
\newblock Journal of Dynamics and Differential Equations, Vol. 8, Issue 1: 1-33, 1996.

\bibitem{T}
R. Temam.
\newblock {\em Navier-Stokes Equations: Theory and Numerical Analysis}.
\newblock AMS Chelsea, Providence, Rhode Island, 2000.

\end{thebibliography}
\end{document}